\documentclass[12pt, reqno]{article}
\usepackage{enumerate,amsmath,amssymb,bm,ascmac,amsthm,url}
\usepackage{tikz}
\usepackage[margin=30truemm]{geometry}
\usepackage{here}

\theoremstyle{definition}
\newtheorem{thm}{Theorem}[section]
\newtheorem{thm*}{Theorem}

\newtheorem{defi*}{Definition}
\newtheorem{lem}[thm]{Lemma}
\newtheorem{lem*}{Lemma}
\newtheorem{pro}[thm]{Proposition}
\newtheorem{pro*}{Proposition}

\newcommand{\MC}[1]{\mathcal{#1}}
\newcommand{\MB}[1]{\mathbb{#1}}

\newcommand{\BM}[1]{{\bm #1}}

\newcommand{\G}{\Gamma}

\newcommand{\Span}[1]{\left \langle #1 \right \rangle}

\makeatletter

\@addtoreset{equation}{section}
\makeatother

\title
{
Perfect state transfer in Grover walks between states associated to vertices of a graph
}

\author{Sho Kubota\thanks{
Department of Applied Mathematics, Faculty of Engineering, Yokohama National University,
Hodogaya, Yokohama 240-8501, Japan. \texttt{kubota-sho-bp@ynu.ac.jp}
}
\and
Etsuo Segawa\thanks{
Graduate School of Environment Information Sciences, Yokohama National University,
Hodogaya, Yokohama, 240-8501, Japan. \texttt{segawa-etsuo-tb@ynu.ac.jp}
}
}

\date{}

\begin{document}
\maketitle
\begin{abstract}
We study perfect state transfer in Grover walks,
which are typical discrete-time quantum walk models.
In particular, we focus on states associated to vertices of a graph.
We call such states vertex type states. 
Perfect state transfer between vertex type states can be studied via Chebyshev polynomials.
We derive a necessary condition on eigenvalues of a graph
for perfect state transfer between vertex type states to occur.
In addition,
we perfectly determine the complete multipartite graphs whose partite sets are the same size on which perfect state transfer occurs between vertex type states,
together with the time.
\vspace{8pt} 
\end{abstract}

\section{Introduction}
Quantum state transfer is an important task in quantum information processing.
Quantum key distribution is one of its applications \cite{BB, E}.
A series of quantum search algorithms can be seen as another application,
which is state transfer from a uniform state to a desired state.
If state transfer is realized with probability 1,
it is said to be ``perfect".
Perfect state transfer has been extensively studied
in both physics \cite{Kay} and mathematics \cite{G2012, G2012-2}.
In mathematics,
it has been studied in continuous-time quantum walks for a relatively long time.
We refer \cite{CS, S} for example.
Although not as numerous as in continuous-time quantum walks, we can find studies of perfect state transfer also in discrete-time quantum walks \cite{BPAK, KT, KW, LCETK, SS2016, SS2017, YG, Z2019, ZQBLX}.
These studies are briefly summarized in Section~1 of \cite{Z2019}.
Recently,
state transfer in discrete-time quantum walks has been studied
also from the viewpoint of algebraic graph theory.
Zhan \cite{Z2019} has constructed an infinite family of 4-regular circulant graphs
on which perfect state transfer occurs.
Chan and Zhan \cite{CZ} have established the theory for pretty good state transfer in discrete-time quantum walks.

In this paper, we consider perfect state transfer in Grover walks,
which are typical discrete-time quantum walk models.
In particular, we focus on perfect state transfer between states associated to vertices of a graph.
We call such states {\it vertex type states}.
We provide a necessary condition on eigenvalues of a graph such that perfect state transfer occurs.
There are two main theorems in this paper.
One is a general fact on arbitrary graphs,
and the other is a precise result on specific graphs.
We first state these theorems.
Readers who want detailed terminologies and definitions can find them
in later sections.
The first main theorem is a condition on eigenvalues of a graph
for perfect state transfer to occur between vertex type states.

\begin{thm}
{\it
Let $\G$ be a graph with the discriminant $P$,
and let $x,y$ be vertices of $\G$.
If perfect state transfer occurs from $d^*\BM{e}_x$ to $d^*\BM{e}_y$ at time $\tau$,
then $T_{\tau}(\lambda) = \pm 1$ holds for any $\lambda \in \Theta_{P}(\BM{e}_x)$,
where $d$ is the boundary matrix,
$T_n(x)$ is the Chebyshev polynomial of the first kind,
and $\Theta_{P}(\BM{e}_x)$ is the eigenvalue support of the unit vector $\BM{e}_x$
with respect to $P$.
}
\end{thm}

The second main theorem is a necessary and sufficient condition for perfect state transfer to occur
on the complete $r$-partite graph $\G_{r,m}$ whose partite sets have the same size $m$.
See the beginning of Section~\ref{0813-3} for the notation of vertices in $\G_{r,m}$.

\begin{thm}
{\it
If perfect state transfer occurs on $\G_{r,m}$
from $d^*\BM{e}_x$ to $d^*\BM{e}_y$ for some vertices $x,y$ in $\G_{r,m}$,
then $r = 2,3$.
Moreover, we have the following:
\begin{enumerate}[(i)]
\item Perfect state transfer occurs
from $d^* \BM{e}_{v_{1}^{(1)}}$ to another vertex type state on $\G_{2,m}$
at time $\tau \in \{1,2,3\}$
if and only if $m=2$ and $\tau = 2$.
In this case,
the state $d^* \BM{e}_{v_{1}^{(1)}}$ transfers to $d^* \BM{e}_{v_{2}^{(1)}}$.
\item Perfect state transfer occurs
from $d^* \BM{e}_{v_{1}^{(1)}}$ to another vertex type state on $\G_{3,m}$
at time $\tau \in \{1,2, \dots, 11\}$
if and only if $m=2$ and $\tau = 6$.
In this case,
the state $d^* \BM{e}_{v_{1}^{(1)}}$ transfers to $d^* \BM{e}_{v_{2}^{(1)}}$.
\end{enumerate}
}
\end{thm}

There is a remark in the above theorem.
We restrict times to consider,
but this does not diminish worth of the theorem.
This is because both $\G_{2,m}$ and $\G_{3,m}$ are periodic graphs \cite{HKSS2017},
i.e., $U(\G_{2,m})^4 = I$ and $U(\G_{3,m})^{12} = I$,
where $U(\G)$ is the time evolution matrix of a graph $\G$.
Thus, it suffices to study only times less than their periods.

This paper is organized as follows.
Section~\ref{0813-1} is a preliminary part.
We define terminologies associated to graphs and Grover walks.
We then define perfect state transfer and
discuss a motivation for considering vertex type states.
In Section~\ref{0813-2}, we derive the first main theorem.
Perfect state transfer between vertex type states is related to the Chebyshev polynomials.
The connection between the both is the main subject.
After the derivation of the first main theorem,
we provide its useful application.
This is also related to the study of periodicity of quantum walks.
In Section~\ref{0813-3}, we discuss perfect state transfer on complete multipartite graphs whose partite sets are the same size.
We show that if perfect state transfer occurs between vertex type states on the graphs, it is only in complete bipartite graphs or complete tripartite graphs.

\section{Preliminaries} \label{0813-1}
See \cite{GR} for basic terminologies related to graphs.
Let $\G =(V, E)$ be a finite simple and connected graph with the vertex set $V$ and the edge set $E$.
Define $\MC{A} = \MC{A}(\G)=\{ (x, y), (y, x) \mid \{x, y\} \in E \}$,
which is the set of the symmetric arcs of $\G$.
The origin and terminus of $a=(x, y) \in \MC{A}$ are denoted by $o(a), t(a)$, respectively.
We write the inverse arc of $a$ as $a^{-1}$.

\subsection{Grover walks and related matrices}
Let $\G = (V, E)$ be a graph.
We define several matrices on Grover walks.
The {\it boundary matrix} $d = d(\G) \in \MB{C}^{V \times \MC{A}}$ is defined by
$d_{x,a} = \frac{1}{\sqrt{\deg x}} \delta_{x, t(a)}$,
where $\delta_{a,b}$ is the Kronecker delta.
Note that
\begin{equation} \label{0803-2}
dd^* = I,
\end{equation}
where $I$ is the identity matrix.
The {\it shift matrix} $S = S(\G) \in \MB{C}^{\MC{A} \times \MC{A}}$
is defined by $S_{a, b} = \delta_{a,b^{-1}}$.
Clearly,
\begin{equation} \label{0820-1}
S^2 = I.
\end{equation}
Define the {\it time evolution matrix} $U = U(\G) \in \MB{C}^{\MC{A} \times \MC{A}}$
by $U = S(2d^*d-I)$.
Quantum walks defined by $U$ is called {\it Grover walks}.
The {\it discriminant} $P=P(\G) \in \MB{C}^{V \times V}$ is defined by $P = dSd^*$.

The discriminant is strongly related to the adjacency matrix of a graph,
and facilitates analysis of Grover walks.
The {\it adjacency matrix} $A=A(\G) \in \MB{C}^{V \times V}$ of a graph $\G = (V, E)$ is defined by
\[ A_{x,y} = \begin{cases}
1 \qquad &\text{$\{ x,y \} \in E$,} \\
0 \qquad &\text{otherwise.}
\end{cases} \]
Direct calculation leads to the following.

\begin{lem}
{\it
Let $\G$ be a $k$-regular graph,
and let $A$ and $P$ be the adjacency matrix and the discriminant of $\G$, respectively.
Then we have $P = \frac{1}{k}A$.
Therefore, the absolute values of eigenvalues of $P$ does not exceed $1$.
}
\end{lem}

We omit a proof.
A more general claim and its proof can be found
in Theorem~3.1 of \cite{KST} and Proposition~3.3 of \cite{KST}.

Since the discriminant $P$ is a normal matrix,
it has a spectral decomposition.
Let $\lambda_1, \dots, \lambda_s$ be the distinct eigenvalues of $P$,
and let $E_i$ be the projection matrices associated to the eigenvalue $\lambda_i$.
Then we have
$P = \sum_{i=1}^{s} \lambda_i E_i$,
and the projection matrices satisfy
\begin{equation} \label{0804-2}
E_i E_j = \delta_{i,j} E_i,
\end{equation}
\begin{equation} \label{0806-2}
E_i^* = E_i,
\end{equation}
and
\begin{equation} \label{0806-3}
\sum_{i=1}^s E_i = I.
\end{equation}
See \cite{M} for the spectral decomposition of matrices.
In this study, we use polynomials of $P$.
The following can be obtained by using Equality~(\ref{0804-2}).

\begin{lem} \label{0806-1}
{\it
Let $M$ be a normal matrix with the spectral decomposition $M = \sum_{i=1}^{s} \lambda_i E_i$,
and let $p(x)$ be a polynomial.
Then we have $p(M) = \sum_{i=1}^{s} p(\lambda_i) E_i$.
}
\end{lem}

\subsection{Perfect state transfer}
Let $\G$ be a graph,
and let $U = U(\G)$ be the time evolution matrix.
A vector $\Phi \in \MB{C}^{\MC{A}}$ is a {\it state} if $|| \Phi || = 1$.
We say that {\it perfect state transfer} occurs from a state $\Phi$ to a state $\Psi$ at time $\tau$
if there exists $\gamma \in \MB{C}$ with norm one such that $U^{\tau}\Phi = \gamma \Psi$.
As mentioned in Section~4 of \cite{CZ},
the occurrence of perfect state transfer can be restated as follows.

\begin{lem}[\cite{CZ}] \label{0820-2}
{\it
Let $\G$ be a graph,
and let $U = U(\G)$ be the time evolution matrix.
Perfect state transfer occurs from a state $\Phi$ to a state $\Psi$ at time $\tau$
if and only if $|\Span{ U^{\tau}\Phi, \Psi }| = 1$.
}
\end{lem}

If no particular restriction is given to states,
we can find various perfect state transfer, including trivial cases.
For example,
if a state $\Phi$ is an eigenvector of $U$ associated to an eigenvalue $\lambda$,
then we have $U \Phi = \lambda \Phi$.
Mathematically, perfect state transfer certainly occurs from $\Phi$ to $\Phi$ at time $1$,
but it would not make sense for such a state $\Phi$.
Moreover, the state does not change.
On the other hand, perfect state transfer always occurs
when $\deg t(a) \in \{1,2\}$ for some $a \in \MC{A}$ in Grover walks.
We can find its explanation in Lemma~6.1 of \cite{KSeYa}.
This is also perfect state transfer between different states, but it is still trivial.
From these points of view,
we should put some restrictions on states,
and the restricted states should be meaningful.
Let us consider states associated to vertices of a graph.

Let $\G$ be a graph with the vertex set $V$.
A state $\Phi$ is said to be {\it vertex type}
if there exists a vertex $x \in V$ such that $\Phi = d^* \BM{e}_{x}$,
where $\BM{e}_{x} \in \MB{C}^{V}$ is the unit vector defined by $(\BM{e}_{x})_z = \delta_{x,z}$.
We denote by $\MC{X}$ the set of all vertex type states
i.e., $\MC{X} = \{ d^*\BM{e}_{x} \mid x \in V \}$.
A vertex type state represents a state that has values only on arcs reaching toward some vertex.
Indeed,
letting $\BM{e}_{a} \in \MB{C}^{\MC{A}}$ be the unit vector defined by $(\BM{e}_{a})_z = \delta_{a,z}$,
we have
\[ d^* \BM{e}_{x} = \frac{1}{\sqrt{\deg x}}\sum_{\substack{a \in \MC{A} \\ t(a) = x}} \BM{e}_{a}. \]
This can be visually represented as shown in Figure~\ref{0907-1}.
A vertex type state corresponds to the situation that
a quantum walker is only on a particular vertex
in the model where a walker moves on vertices.
We study perfect state transfer between states in $\MC{X}$.

\begin{figure}[ht]
\begin{center}
\begin{tikzpicture}
[scale = 0.7,
line width = 0.8pt,
v/.style = {circle, fill = black, inner sep = 0.8mm},u/.style = {circle, fill = white, inner sep = 0.1mm}]
  \node[v] (1) at (2, 0) {};
  \node[v] (2) at (0, 2) {};
  \node[v] (3) at (-2, 0) {};
  \node[v] (4) at (0, 4.6) {};
  \node[u] (10) at (0.4, 2.4) {$x$};
  \node[u] (1012) at (1.6, 1.65) {\textcolor{blue}{$\frac{1}{\sqrt{3}}$}};
  \node[u] (1032) at (-1.7, 1.65) {\textcolor{blue}{$\frac{1}{\sqrt{3}}$}};
  \node[u] (1032) at (-0.6, 3.3) {\textcolor{blue}{$\frac{1}{\sqrt{3}}$}};
  \draw[-] (1) to (2);
  \draw[-] (2) to (3);
  \draw[-] (4) to (2);  
  \node[u] (11) at (2, 0.3) {};
  \node[u] (22) at (0.3, 2) {};
  \draw[draw= blue,->] (11) to (22);
  \node[u] (23) at (-0.3, 2) {};
  \node[u] (33) at (-2, 0.3) {};
  \draw[draw= blue,->] (33) to (23);
  \node[u] (a) at (-0.15, 4.4) {};
  \node[u] (b) at (-0.15, 2.2) {};
  \draw[draw= blue,->] (a) to (b);
\end{tikzpicture}
\caption{A vertex type state $d^* \BM{e}_x$ with $\deg x = 3$} \label{0907-1}
\end{center}
\end{figure}
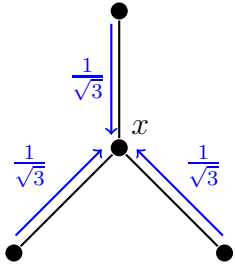

\section{Chebyshev polynomials of the first kind} \label{0813-2}

In Grover walks,
perfect state transfer between vertex type states and Chebyshev polynomials are
closely related to each other.
We first introduce the polynomials.
The {\it Chebyshev polynomial of the first kind}, denoted by $T_n(x)$,
is the polynomial defined by $T_0(x) = 1$, $T_1(x) = x$ and
$T_{n}(x) = 2xT_{n-1}(x) - T_{n-2}(x)$ for $n \geq 2$.
It is well-known that 
\begin{equation} \label{0802-1}
T_n(\cos \theta) = \cos (n \theta).
\end{equation}
This implies
\begin{equation} \label{0723-4}
|T_n(x)| \leq 1
\end{equation}
for $|x| \leq 1$.

\begin{lem} \label{0723-5}
{\it
Let $\G$ be a graph with the time evolution matrix $U$ and the discriminant $P$.
Then we have $d U^{\tau} d^{*} = T_{\tau}(P)$ for $\tau \in \MB{N} \cup \{0\}$.
}
\end{lem}

\begin{proof}
We first claim that
\begin{equation} \label{0723-1}
d U^{\tau} d^{*} = d U^{\tau - 1} S d^{*}
\end{equation}
for $\tau \geq 1$.
Indeed,
\begin{align*}
d U^{\tau} d^{*}
&= d U^{\tau-1} S (2d^* d - I) d^{*} \\
&= 2 d U^{\tau-1} S d^* d d^* -  d U^{\tau-1} S d^{*} \\
&= 2 d U^{\tau-1} S d^* -  d U^{\tau-1} S d^{*} \tag{by~(\ref{0803-2})} \\
&= d U^{\tau-1} S d^{*}.
\end{align*}

We next show the statement by induction on $\tau$.
Equality~(\ref{0803-2}) derives $d U^{0} d^{*} = I = T_{0}(P)$.
Equality~(\ref{0723-1}) leads to $d U d^{*} = d S d^{*} = P = T_{1}(P)$.
Assume that the claim holds up to $\tau - 1$ for $\tau \geq 2$.
We have
\begin{align*}
d U^{\tau} d^{*} &= d U^{\tau - 1} S d^{*} \tag{by (\ref{0723-1})} \\
&= d U^{\tau - 2} \{S (2 d^* d - I) \} S d^{*} \\
&= 2 d U^{\tau - 2} S d^* P - d U^{\tau - 2} d^* \tag{by (\ref{0820-1})} \\
&= 2 d U^{\tau - 1} d^* P - T_{\tau - 2}(P) \tag{by (\ref{0723-1})} \\
&= 2P T_{\tau - 1}(P) - T_{\tau - 2}(P) \\
&= T_{\tau}(P).
\end{align*}
The desired equality is obtained for any $\tau \in \MB{N} \cup \{0\}$.
\end{proof}

Let $M \in \MB{C}^{n \times n}$ be a normal matrix
with the spectral decomposition $M = \sum_{i=1}^{s} \lambda_i E_i$.
We denote by $\sigma(M)$ the set of the distinct eigenvalues of $M$.
Remark that $\sigma(M)$ is simply a set, so the multiplicities of the eigenvalues are ignored.
For a vector $x \in \MB{C}^n$, we define $\Theta_M(x) = \{ \lambda_i \in \sigma(M) \mid E_i x \neq 0 \}$.
This is called the {\it eigenvalue support} of the vector $x$ with respect to $M$.

\begin{lem} \label{0723-6}
{\it
Let $\G$ be a graph with the discriminant $P$.
For a vertex $x \in V(\G)$ and $\tau \in \MB{N} \cup \{0\}$,
we have $|| T_{\tau}(P) \BM{e}_x || \leq 1$.
The equality holds if and only if $T_{\tau}(\lambda) = \pm 1$ for any $\lambda \in \Theta_P(\BM{e}_x)$.
}
\end{lem}
\begin{proof}
Let the spectral decomposition of $P$ be
\[ P = \sum_{\lambda \in \sigma(P)} \lambda E_{\lambda}. \]
We note that $\sigma(P) \subset \MB{R}$ since $P$ is a symmetric matrix.
We have
\begin{align*}
|| T_{\tau}(P) \BM{e}_x ||^2
&= \Span{ T_{\tau}(P) \BM{e}_x, T_{\tau}(P) \BM{e}_x } \\
&= \Span{\sum_{\lambda \in \sigma(P)}T_{\tau}(\lambda) E_{\lambda} \BM{e}_x, \sum_{\lambda \in \sigma(P)}T_{\tau}(\lambda) E_{\lambda} \BM{e}_x} \tag{by Lemma~\ref{0806-1}} \\
&= \sum_{\lambda \in \sigma(P)} \Span{T_{\tau}(\lambda) E_{\lambda} \BM{e}_x, T_{\tau}(\lambda) E_{\lambda} \BM{e}_x} \tag{by (\ref{0806-2}) and (\ref{0804-2})} \\
&= \sum_{\lambda \in \sigma(P)} T_{\tau}(\lambda)^2 \Span{E_{\lambda} \BM{e}_x, E_{\lambda} \BM{e}_x} \\
&\leq \sum_{\lambda \in \sigma(P)} \Span{E_{\lambda} \BM{e}_x, E_{\lambda} \BM{e}_x} \tag{by (\ref{0723-4})}\\
&= \Span{\BM{e}_x, \sum_{\lambda \in \sigma(P)} E_{\lambda} \BM{e}_x} \tag{by (\ref{0806-2}) and (\ref{0804-2})}\\
&= \Span{\BM{e}_x, \BM{e}_x} \tag{by (\ref{0806-3})} \\
&= 1.
\end{align*}
The equality holds if and only if $T_{\tau}(\lambda)^2 \Span{E_{\lambda} \BM{e}_x, E_{\lambda} \BM{e}_x} = \Span{E_{\lambda} \BM{e}_x, E_{\lambda} \BM{e}_x}$ for any $\lambda \in \sigma(P)$.
This is equivalent to the condition that $T_{\tau}(\lambda) = \pm 1$ for any $\lambda \in \Theta_P(\BM{e}_x)$
because $E_{\lambda} \BM{e}_x = 0$ for $\lambda \in \sigma(P) \setminus \Theta_P(\BM{e}_x)$.
\end{proof}

These two lemmas lead to a useful necessary condition for perfect state transfer to occur.

\begin{thm} \label{0724-3}
{\it
Let $\G$ be a graph with the discriminant $P$,
and let $x,y$ be vertices of $\G$.
If perfect state transfer occurs from $d^*\BM{e}_x$ to $d^*\BM{e}_y$ at time $\tau$,
then $T_{\tau}(\lambda) = \pm 1$ holds for any $\lambda \in \Theta_{P}(\BM{e}_x)$.
}
\end{thm}
\begin{proof}
Indeed, if perfect state transfer occurs, then
\begin{align*}
1 &=  \left | \Span{ U^{\tau} d^*\BM{e}_x, d^*\BM{e}_y} \right | \tag{by Lemma~\ref{0820-2}} \\
&= \left | \Span{ dU^{\tau} d^*\BM{e}_x, \BM{e}_y} \right | \\
&= \left | \Span{ T_{\tau}(P) \BM{e}_x, \BM{e}_y}\right | \tag{by Lemma~\ref{0723-5}} \\
&\leq ||T_{\tau}(P) \BM{e}_x|| \cdot ||\BM{e}_y|| \\
&\leq 1. \tag{by Lemma~\ref{0723-6}}
\end{align*}
Thus, all inequalities are tight.
In particular, we have $||T_{\tau}(P) \BM{e}_x|| = 1$.
From Lemma~\ref{0723-6},
we have $T_{\tau}(\lambda) = \pm 1$ for any $\lambda \in \Theta_{P}(\BM{e}_x)$.
\end{proof}

The above theorem provides information on eigenvalues for perfect state transfer to occur.
In particular, when the discriminant has a rational eigenvalue,
then the eigenvalue is quite restricted.
The following are lemmas to describe this.

\begin{lem} \label{0724-4}
{\it
We have
\[ \bigcup_{\tau \in \MB{N}} \{ x \in [-1,1] \mid T_{\tau}(x) = \pm 1 \} \subset \cos \MB{Q} \pi. \]
}
\end{lem}
\begin{proof}
Take an element $x$ from the left hand side.
There exists $\tau \in \MB{N}$ such that $T_{\tau}(x) =\pm 1$.
We can put $x=\cos \theta$ for some $\theta \in [0, \pi]$ since $x \in [-1,1]$.
Thus we have $\cos (\tau \theta) = \pm 1$.
This is equivalent to $\theta \in \{ \frac{j}{\tau}\pi \mid j \in \{0,1,\dots, \tau\} \} \subset \MB{Q}\pi$.
\end{proof}

A complex number $\alpha$ is said to be an {\it algebraic integer}
if there exists a monic polynomial $p(x) \in \MB{Z}[x]$ such that $p(\alpha) = 0$.
Let $\Omega$ denote the set of algebraic integers.
We provide some basic facts on the algebraic integers.

\begin{pro}[\cite{J, L}] \label{0727-1}
{\it
We have the following.
\begin{enumerate}[(i)]
\item $\Omega$ is a subring of $\MB{C}$; and
\item $\Omega \cap \MB{Q} = \MB{Z}$.
\end{enumerate}
}
\end{pro}

\begin{lem} \label{0724-5}
{\it
We have $\cos \MB{Q}\pi \cap \MB{Q} = \{ \pm 1, \pm \frac{1}{2}, 0 \}$.
}
\end{lem}

\begin{proof}
It is clear that the right hand side is included in the left hand side.
We show the reverse inclusion.
Without loss of generality,
we take $\cos \frac{2p}{q} \pi \in  \cos \MB{Q}\pi \cap \MB{Q}$,
where $p,q \in \MB{Z}$.
Let $\zeta = e^{\frac{2p}{q} \pi i}$.
Since $\zeta^q = \zeta^{-q} = 1$, we have $\zeta, \zeta^{-1} \in \Omega$.
By the assumption and Proposition~\ref{0727-1}~(i),
we have $\Omega \ni \zeta + \zeta^{-1} = 2 \cos \frac{2p}{q} \in \MB{Q}$.
Proposition~\ref{0727-1}~(ii) derives $2 \cos \frac{2p}{q} \in \Omega \cap \MB{Q} = \MB{Z}$.
Therefore, we have $\cos \frac{2p}{q}\pi \in \{ \pm 1, \pm \frac{1}{2}, 0 \}$.
\end{proof}

For example,
we consider a situation that a rational eigenvalue $\lambda$ of the discriminant
is in the eigenvalue support of some vertex type state.
If perfect state transfer occurs between vertex type states,
then the eigenvalue $\lambda$ has to be $\pm 1, \pm \frac{1}{2}$ or $0$.
Indeed,
Theorem~\ref{0724-3} and Lemma~\ref{0724-4} yield that
\[ \MB{Q} \ni \lambda \in \bigcup_{\tau \in \MB{N}} \{ x \in [-1,1] \mid T_{\tau}(x) = \pm 1 \} \subset \cos \MB{Q} \pi, \]
and hence Lemma~\ref{0724-5} derives
$\lambda \in \cos \MB{Q}\pi \cap \MB{Q} = \{ \pm 1, \pm \frac{1}{2}, 0 \}$.

We note that a similar claim was provided in the context of periodicity of Grover walks.
In the proof of Theorem~1.2 of \cite{Y2019},
Yoshie claimed that if the discriminant has a rational eigenvalue,
then it must be one of the five listed above.
We also find approaches that
rational eigenvalues restrict candidates of periodic graphs in \cite{HKSS2017}.
Considerations of quantum walks with algebraic integers were also given in \cite{KKKS, K, SMA}.

Using the facts presented in this section,
we determine perfect state transfer between vertex type states on the complete multipartite graphs.

\section{Perfect state transfer on the complete multipartite graphs} \label{0813-3}

Let $r, m \geq 2$ and let $V_1, V_2, \dots, V_r$ be sets of $m$ elements
such that $V_i \cap V_j = \emptyset$ for distinct $i,j \in \{1, \dots, r\}$.
Define the graph $\G_{r,m}$ by the vertex set
\[ V(\G_{r,m}) = \bigcup_{i=1}^r V_i, \]
where $x \in V_i$ and $y \in V_j$ are adjacent if and only if $i \neq j$.
The graph $\G_{r,m}$ is a kind of complete multipartite graph,
sometimes written as $K_{m, \dots, m}$, where $m$ appears $r$ times.
In other words, the graph $\G_{r,m}$ is the complement of $r K_m$.
We display as $V_j = \{ v_1^{(j)}, v_2^{(j)}, \dots, v_m^{(j)} \}$ for $j \in \{1,2,\dots, r\}$.
Note that the adjacency matrix of $\G_{r,m}$ is $A(\G_{r,m}) = A(K_r) \otimes J_m$,
where $J_{m}$ is the all-ones matrix of size $m$.
Since $\G_{r,m}$ is $m(r-1)$-regular,
we have $P(\G_{r,m}) = \frac{1}{m(r-1)}A(\G_{r,m}) = (\frac{1}{r-1}A(K_r)) \otimes (\frac{1}{m} J_m)$.
The spectral decompositions of $\frac{1}{r-1}A(K_r)$ and $\frac{1}{m} J_m$ are respectively
\begin{align}
\frac{1}{r-1}A(K_r) &= 1 \cdot E_1 + \left(- \frac{1}{r-1} \right) E_2, \label{0723-2} \\
\frac{1}{m} J_m &= 1 \cdot F_1 + 0 \cdot F_2, \label{0723-3}
\end{align}
where $E_1 = \frac{1}{r} J_r$, $E_2 = I_r - \frac{1}{r} J_r $,
$F_1 = \frac{1}{m} J_m$ and $F_2 = I_m - \frac{1}{m} J_m$.

\begin{lem} \label{0724-2}
{\it
Let $P = P(\G_{r,m})$.
Then we have $\Theta_P(\BM{e}_{v_{i}^{(j)}}) = \sigma(P)$ for any vertex $v_{i}^{(j)}$ in $\G_{r,m}$.
}
\end{lem}

\begin{proof}
It is sufficient to show that $\Theta_P(\BM{e}_{v_{i}^{(j)}}) \supset \sigma(P)$.
By (\ref{0723-2}) and (\ref{0723-3}),
we have
\begin{equation} \label{0724-1}
P = 1 \cdot (E_1 \otimes F_1) + \left(- \frac{1}{r-1} \right) (E_2 \otimes F_1) + 0 \cdot (I \otimes F_2)
\end{equation}
as the spectral decomposition of $P$.
Since $\BM{e}_{v_{i}^{(j)}} = \BM{e}_j \otimes \BM{e}_i$,
we have $(E_1 \otimes F_1)(\BM{e}_j \otimes \BM{e}_i) \neq 0$,
$(E_2 \otimes F_1)(\BM{e}_j \otimes \BM{e}_i) \neq 0$
and $(I \otimes F_2)(\BM{e}_j \otimes \BM{e}_i) \neq 0$.
This implies the statement.
\end{proof}

First, we show that $r = 2,3$ if perfect state transfer occurs between vertex type states on $\G_{r,m}$.

\begin{lem} \label{0803-1}
{\it
If perfect state transfer occurs on $\G_{r,m}$
from $d^*\BM{e}_x$ to $d^*\BM{e}_y$ for some vertices $x,y$ in $\G_{r,m}$,
then $r = 2,3$.
}
\end{lem}
\begin{proof}
Let $P$ be the discriminant of $\G_{r,m}$.
Equality~(\ref{0724-1}) implies that the distinct eigenvalues of $P$ are $1, 0 , -\frac{1}{r-1}$.
We have $\Theta_{P}(\BM{e}_x) = \sigma(P)$ by Lemma~\ref{0724-2}.
From Theorem~\ref{0724-3},
we have $T_{\tau}(-\frac{1}{r-1}) = \pm 1$ for some $\tau \in \MB{N}$.
Lemma~\ref{0724-4} derives $-\frac{1}{r-1} \in \cos \MB{Q} \pi$,
and $-\frac{1}{r-1}$ itself is a rational number.
By Lemma~\ref{0724-5}, we have $-\frac{1}{r-1} \in \{\pm 1, \pm \frac{1}{2} , 0 \}$.
This implies $r = 2,3$.
\end{proof}


From the above lemma,
it is sufficient to discuss each case of $r=2$ and $r=3$.
Before discussing the specific cases,
we state the following which is common to both.

\begin{lem} \label{0730-1}
{\it
With the above notation, if
\[ \left | \Span{ (I \otimes (F_1 - F_2)) \BM{e}_{v_{1}^{(1)}}, \BM{e}_{v_{i}^{(j)}}} \right | = 1, \]
then we have $m=2$ and $v_{i}^{(j)} = v_2^{(1)}$.
}
\end{lem}
\begin{proof}
We note that $\BM{e}_{v_{i}^{(j)}} = \BM{e}_j \otimes \BM{e}_i$.
From the assumption, we obtain
\begin{align*}
1 &= \left | \Span{ (I \otimes (F_1 - F_2)) \BM{e}_{v_{1}^{(1)}}, \BM{e}_{v_{i}^{(j)}}} \right | \\
&= \left | \Span{ (I \otimes (F_1 - F_2)) (\BM{e}_1 \otimes \BM{e}_1), (\BM{e}_j \otimes \BM{e}_i)} \right | \\
&= \left | \Span{ (\BM{e}_1 \otimes (F_1 - F_2)\BM{e}_1), (\BM{e}_j \otimes \BM{e}_i)} \right | \\
&= \left | \BM{e}_1^{\top}\BM{e}_j \cdot ((F_1 - F_2)\BM{e}_1)^{\top} \BM{e}_i \right | \\
&= \left | \delta_{1,j} \left( \frac{2}{m} - \delta_{1,i} \right) \right |.
\end{align*}
The above equality implies $j=1$.
In addition, we have $i \neq 1$ since $m \geq 2$. 
We now have $\frac{2}{m} = 1$, i.e., $m=2$.
This implies that $i=2$, so we have $v_{i}^{(j)} = v_2^{(1)}$.
\end{proof}

\subsection{The case $r = 2$}
We discuss the graph $\G_{2,m}$.
Higuchi et al \cite{HKSS2017} have shown that
$\G_{2,m} = K_{m,m}$ is a periodic graph with period $4$,
i.e., $U(K_{m,m})^4 = I$.
Thus, we consider the time $\tau \in \{1,2,3\}$.
By the symmetry of $K_{m,m}$,
we may assume that the initial state is $d^* \BM{e}_{v_{1}^{(1)}}$ without loss of generality.

\begin{thm}
{\it
Perfect state transfer occurs from $d^* \BM{e}_{v_{1}^{(1)}}$ to another vertex type state on $\G_{2,m}$
at time $\tau \in \{1,2,3\}$
if and only if $m=2$ and $\tau = 2$.
In this case,
the state $d^* \BM{e}_{v_{1}^{(1)}}$ transfers to $d^* \BM{e}_{v_{2}^{(1)}}$.
}
\end{thm}

\begin{proof}
Suppose that perfect state transfer occurs from $d^* \BM{e}_{v_{1}^{(1)}}$ to $d^* \BM{e}_{v_{i}^{(j)}}$.
We first determine the time at which perfect state transfer can occur.
Let $P= P(\G_{2,m})$ and $U= U(\G_{2,m})$.
By Theorem~\ref{0724-3} and $0 \in \Theta_{P}(\BM{e}_{v_{1}^{(1)}})$,
we have $T_{\tau}(0) = \pm 1$, and hence $\tau = 2$.
In addition,
\begin{align*}
T_2(P) &= T_2(1 \cdot (E_1 \otimes F_1) + (-1) \cdot (E_2 \otimes F_1) + 0 \cdot (I \otimes F_2)) \tag{by (\ref{0724-1})} \\
&= (E_1 \otimes F_1) + (E_2 \otimes F_1) - (I \otimes F_2) \tag{by Lemma~\ref{0806-1}} \\
&= I \otimes (F_1 - F_2).
\end{align*}
Thus,
\begin{align*}
1 &= \left | \Span{ U^2 d^* \BM{e}_{v_{1}^{(1)}}, d^* \BM{e}_{v_{i}^{(j)}} } \right | \tag{by Lemma~\ref{0820-2}} \\
&= \left | \Span{ T_2(P) \BM{e}_{v_{1}^{(1)}}, \BM{e}_{v_{i}^{(j)}}} \right | \tag{by Lemma~\ref{0723-5}} \\
&= \left | \Span{ (I \otimes (F_1 - F_2)) \BM{e}_{v_{1}^{(1)}}, \BM{e}_{v_{i}^{(j)}}} \right |.
\end{align*}
By Lemma~\ref{0730-1}, we have $m=2$ and $v_{i}^{(j)} = v_2^{(1)}$.
Conversely,
we suppose that $m=2$ and $\tau = 2$.
Direct calculation derives $T_2(P) \BM{e}_{v_{1}^{(1)}} = \BM{e}_{v_{2}^{(1)}}$.
This implies that 
$\Span{ U^2 d^* \BM{e}_{v_{1}^{(1)}}, d^* \BM{e}_{v_{2}^{(1)}} } = 1$.
\end{proof}

We note that $\G_{2,2}$ is isomorphic to the 4-cycle $C_4$.
As explained in \cite{KSeYa},
if the degree of the terminus of an arc is 2,
the action of the time evolution matrix can be understood visually.
We supplement this.
Let $\G$ be a graph with $U = U(\G)$, and let $a \in \MC{A}(\G)$.
We denote by $\BM{e}_a$ the unit vector defined by $(\BM{e}_a)_z = \delta_{a,z}$.
If $\deg t(a) = 2$, then direct calculation derives $U \BM{e}_a = \BM{e}_b$,
where $b$ is the arc in $\{ z \in \MC{A}(\G) \mid t(a) = o(z) \} \setminus \{a^{-1}\}$.
This can be described visually as shown in Figure~\ref{0807-1}.

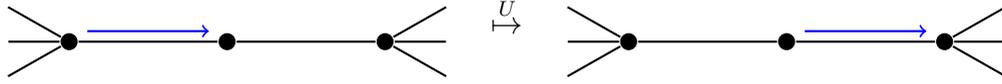
\begin{figure}[ht]
\begin{center}
\begin{tikzpicture}
[scale = 0.7,
line width = 0.8pt,
v/.style = {circle, fill = black, inner sep = 0.8mm},u/.style = {circle, fill = white, inner sep = 0.1mm}]
  \node[u] (1) at (-1.2, 0) {};
  \node[v] (2) at (0, 0) {};
  \node[v] (3) at (3, 0) {};
  \node[u] (4) at (-1.2, 0.68) {};
  \node[u] (5) at (-1.2, -0.68) {};
  \node[v] (6) at (6, 0) {};
  \node[u] (7) at (7.2, 0) {};
  \node[u] (8) at (7.2, 0.68) {};
  \node[u] (9) at (7.2, -0.68) {};
  \node[u] (10) at (0.3, 0.2) {};
  \node[u] (11) at (2.7, 0.2) {};
  \node[u] (12) at (3.3, 0.2) {};
  \node[u] (13) at (5.7, 0.2) {};
  \draw (1) to (2);
  \draw[-] (2) to (4);
  \draw[-] (2) to (3);
  \draw[-] (5) to (2);
  \draw[draw= blue,->] (10) to (11);
  \draw[-] (3) to (6);
  \draw[-] (6) to (7);
  \draw[-] (6) to (8);
  \draw[-] (6) to (9);
\end{tikzpicture}
\raisebox{6mm}{$\quad \overset{U}{\mapsto} \quad$}
\begin{tikzpicture}
[scale = 0.7,
line width = 0.8pt,
v/.style = {circle, fill = black, inner sep = 0.8mm},u/.style = {circle, fill = white, inner sep = 0.1mm}]
  \node[u] (1) at (-1.2, 0) {};
  \node[v] (2) at (0, 0) {};
  \node[v] (3) at (3, 0) {};
  \node[u] (4) at (-1.2, 0.68) {};
  \node[u] (5) at (-1.2, -0.68) {};
  \node[v] (6) at (6, 0) {};
  \node[u] (7) at (7.2, 0) {};
  \node[u] (8) at (7.2, 0.68) {};
  \node[u] (9) at (7.2, -0.68) {};
  \node[u] (10) at (0.3, 0.2) {};
  \node[u] (11) at (2.7, 0.2) {};
  \node[u] (12) at (3.3, 0.2) {};
  \node[u] (13) at (5.7, 0.2) {};
  \draw (4.5,0.7) node[blue]{};
  \draw (1) to (2);
  \draw[-] (2) to (4);
  \draw[-] (2) to (3);
  \draw[-] (5) to (2);
  \draw[draw= blue,->] (12) to (13);
  \draw[-] (3) to (6);
  \draw[-] (6) to (7);
  \draw[-] (6) to (8);
  \draw[-] (6) to (9);
\end{tikzpicture}
\caption{Action of $U$} \label{0807-1}
\end{center}
\end{figure}

Let us try to understand perfect state transfer on $\G_{2,2}$ with this view.
As shown in Figure~\ref{0807-2},
we can see that perfect state transfer occurs between vertex type states by acting $U$ twice.

\begin{figure}[ht]
\begin{center}
\begin{tikzpicture}
[scale = 0.7,
line width = 0.8pt,
v/.style = {circle, fill = black, inner sep = 0.8mm},u/.style = {circle, fill = white, inner sep = 0.1mm}]
  \node[v] (1) at (2, 0) {};
  \node[v] (2) at (0, 2) {};
  \node[v] (3) at (-2, 0) {};
  \node[v] (4) at (0, -2) {};
  \draw[-] (1) to (2);
  \draw[-] (2) to (3);
  \draw[-] (3) to (4);
  \draw[-] (4) to (1);
  \node[u] (11) at (2, 0.3) {};
  \node[u] (22) at (0.3, 2) {};
  \draw[draw= blue,->] (11) to (22);
  \node[u] (23) at (-0.3, 2) {};
  \node[u] (33) at (-2, 0.3) {};
  \draw[draw= red,->] (33) to (23);
\end{tikzpicture}
\raisebox{14mm}{$\quad \overset{U}{\mapsto} \quad$}
\begin{tikzpicture}
[scale = 0.7,
line width = 0.8pt,
v/.style = {circle, fill = black, inner sep = 0.8mm},u/.style = {circle, fill = white, inner sep = 0.1mm}]
  \node[v] (1) at (2, 0) {};
  \node[v] (2) at (0, 2) {};
  \node[v] (3) at (-2, 0) {};
  \node[v] (4) at (0, -2) {};
  \draw[-] (1) to (2);
  \draw[-] (2) to (3);
  \draw[-] (3) to (4);
  \draw[-] (4) to (1);
  \node[u] (11) at (2, 0.3) {};
  \node[u] (22) at (0.3, 2) {};
  \draw[draw= red,->] (22) to (11);
  \node[u] (23) at (-0.3, 2) {};
  \node[u] (33) at (-2, 0.3) {};
  \draw[draw= blue,->] (23) to (33);
\end{tikzpicture}
\raisebox{14mm}{$\quad \overset{U}{\mapsto} \quad$}
\begin{tikzpicture}
[scale = 0.7,
line width = 0.8pt,
v/.style = {circle, fill = black, inner sep = 0.8mm},u/.style = {circle, fill = white, inner sep = 0.1mm}]
  \node[v] (1) at (2, 0) {};
  \node[v] (2) at (0, 2) {};
  \node[v] (3) at (-2, 0) {};
  \node[v] (4) at (0, -2) {};
  \draw[-] (1) to (2);
  \draw[-] (2) to (3);
  \draw[-] (3) to (4);
  \draw[-] (4) to (1);
  \node[u] (33) at (-2, -0.3) {};
  \node[u] (44) at (-0.3, -2) {};
  \draw[draw= blue,->] (33) to (44);
  \node[u] (45) at (0.3, -2) {};
  \node[u] (11) at (2, -0.3) {};
  \draw[draw= red,->] (11) to (45);
\end{tikzpicture}
\caption{Perfect state transfer on $\G_{2,2}$} \label{0807-2}
\end{center}
\end{figure}
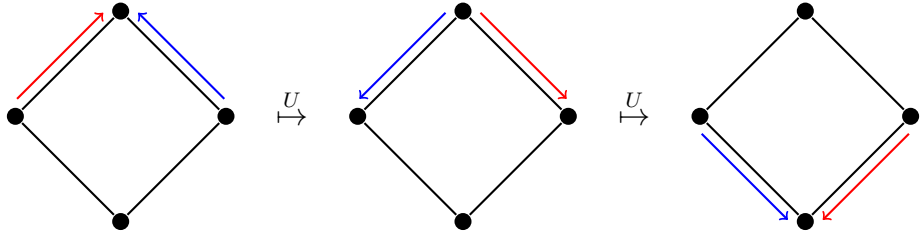

\subsection{The case $r = 3$}
We discuss the graph $\G_{3,m}$.
Remark that $\G_{3,m} = K_{m,m,m}$ is also a periodic graph with period $12$,
i.e., $U(K_{m,m,m})^{12} = I$, as shown in \cite{HKSS2017}.
Thus, we consider the time $\tau \in \{1,2, \dots, 11\}$.
As in the consideration of $r=2$,
we may assume that the initial state is $d^* \BM{e}_{v_{1}^{(1)}}$ without loss of generality.

\begin{thm}
{\it
Perfect state transfer occurs from $d^* \BM{e}_{v_{1}^{(1)}}$ to another vertex type state on $\G_{3,m}$
at time $\tau \in \{1,2, \dots, 11\}$
if and only if $m=2$ and $\tau = 6$.
In this case,
the state $d^* \BM{e}_{v_{1}^{(1)}}$ transfers to $d^* \BM{e}_{v_{2}^{(1)}}$.
}
\end{thm}

\begin{proof}
Suppose that perfect state transfer occurs from $d^* \BM{e}_{v_{1}^{(1)}}$ to $d^* \BM{e}_{v_{i}^{(j)}}$.
We first determine the time at which perfect state transfer can occur.
Let $P= P(\G_{3,m})$ and $U= U(\G_{3,m})$.
By Theorem~\ref{0724-3} and $0 \in \Theta_{P}(\BM{e}_{v_{1}^{(1)}})$,
we have $T_{\tau}(0) = \pm 1$, and hence $\tau$ is even.
Let $\tau = 2s$ for some $s \in \MB{N}$.
Since $-\frac{1}{2} \in \Theta_{P}(\BM{e}_{v_{1}^{(1)}})$,
we have $T_{\tau}(-\frac{1}{2}) = \pm 1$.
From Equality~(\ref{0802-1}),
$\pm 1 = T_{\tau}(-\frac{1}{2}) = \cos \frac{4s}{3}\pi$, which implies $s \in 3\MB{Z}$.
We now have $\tau \in 6\MB{Z}$, i.e., $\tau = 6$.
In addition,
\begin{align*}
T_6(P) &= T_6 \left( 1 \cdot (E_1 \otimes F_1) + \left( -\frac{1}{2} \right) \cdot (E_2 \otimes F_1) + 0 \cdot (I \otimes F_2) \right) \tag{by (\ref{0724-1})} \\
&= (E_1 \otimes F_1) + (E_2 \otimes F_1) - (I \otimes F_2) \tag{by Lemma~\ref{0806-1}} \\
&= I \otimes (F_1 - F_2).
\end{align*}
Thus,
\begin{align*}
1 &= \left | \Span{ U^6 d^* \BM{e}_{v_{1}^{(1)}}, d^* \BM{e}_{v_{i}^{(j)}} } \right | \tag{by Lemma~\ref{0820-2}} \\
&= \left | \Span{ T_6(P) \BM{e}_{v_{1}^{(1)}}, \BM{e}_{v_{i}^{(j)}}} \right | \tag{by Lemma~\ref{0723-5}} \\
&= \left | \Span{ (I \otimes (F_1 - F_2)) \BM{e}_{v_{1}^{(1)}}, \BM{e}_{v_{i}^{(j)}}} \right |.
\end{align*}
By Lemma~\ref{0730-1}, we have $m=2$ and $v_{i}^{(j)} = v_2^{(1)}$.
Conversely,
we suppose that $m=2$ and $\tau = 6$.
Direct calculation leads to $T_6(P) \BM{e}_{v_{1}^{(1)}} = \BM{e}_{v_{2}^{(1)}}$.
This implies that 
$\Span{ U^6 d^* \BM{e}_{v_{1}^{(1)}}, d^* \BM{e}_{v_{2}^{(1)}} } = 1$.
\end{proof}

\subsection{The complete graphs}
In this subsection,
we briefly supplement perfect state transfer in the complete graph.
Formally, the complete graph can be written as $\G_{r,1}$.
Since this graph also has the eigenvalue $-\frac{1}{r-1}$,
the same as in Lemma~\ref{0803-1} holds.
Namely, if perfect state transfer occurs between vertex type states on $\G_{r,1}$, then $r=2,3$.
It is easy to verify that perfect state transfer occurs on $\G_{2,1} = K_2$.
More precisely, $U(K_2) \BM{e}_{a} = \BM{e}_{a^{-1}}$ holds for any $a \in \MC{A}(K_2)$.
On the other hand,
perfect state transfer does not occur from a vertex type state to another one on $\G_{3,1} = K_3$.
This can be understood by the fact \cite{HKSS2017, KSeYa} that the period of $K_3$ is 3
and by visually observing dynamics of Grover walks.
We show in Figure~\ref{0809-1}.

\begin{figure}[ht]
\begin{center}
\begin{tikzpicture}
[scale = 0.6,
line width = 0.8pt,
v/.style = {circle, fill = black, inner sep = 0.8mm},u/.style = {circle, fill = white, inner sep = 0.1mm}]
  \node[v] (1) at (2, 0) {};
  \node[v] (2) at (0, 3.4) {};
  \node[v] (3) at (-2, 0) {};
  \draw[-] (1) to (2);
  \draw[-] (2) to (3);
  \draw[-] (3) to (1);
  \node[u] (11) at (2, 0.35) {};
  \node[u] (22) at (0.25, 3.3) {};
  \draw[draw= blue,->] (11) to (22);
  \node[u] (23) at (-0.25, 3.3) {};
  \node[u] (33) at (-2, 0.35) {};
  \draw[draw= red,->] (33) to (23);
\end{tikzpicture}
\raisebox{12mm}{$\quad \overset{U}{\mapsto} \quad$}
\begin{tikzpicture}
[scale = 0.6,
line width = 0.8pt,
v/.style = {circle, fill = black, inner sep = 0.8mm},u/.style = {circle, fill = white, inner sep = 0.1mm}]
  \node[v] (1) at (2, 0) {};
  \node[v] (2) at (0, 3.4) {};
  \node[v] (3) at (-2, 0) {};
  \draw[-] (1) to (2);
  \draw[-] (2) to (3);
  \draw[-] (3) to (1);
  \node[u] (11) at (2, 0.35) {};
  \node[u] (22) at (0.25, 3.3) {};
  \draw[draw= red,->] (22) to (11);
  \node[u] (23) at (-0.25, 3.3) {};
  \node[u] (33) at (-2, 0.35) {};
  \draw[draw= blue,->] (23) to (33);
\end{tikzpicture}
\raisebox{12mm}{$\quad \overset{U}{\mapsto} \quad$}
\begin{tikzpicture}
[scale = 0.6,
line width = 0.8pt,
v/.style = {circle, fill = black, inner sep = 0.8mm},u/.style = {circle, fill = white, inner sep = 0.1mm}]
  \node[v] (1) at (2, 0) {};
  \node[v] (2) at (0, 3.4) {};
  \node[v] (3) at (-2, 0) {};
  \draw[-] (1) to (2);
  \draw[-] (2) to (3);
  \draw[-] (3) to (1);
  \node[u] (33) at (-1.8, -0.3) {};
  \node[u] (11) at (1.8, -0.3) {};
  \draw[draw= blue,->] (33) to (11);
  \node[u] (332) at (-1.8, -0.55) {};
  \node[u] (112) at (1.8, -0.55) {};
  \draw[draw= red,->] (112) to (332);
\end{tikzpicture}
\raisebox{12mm}{$\quad \overset{U}{\mapsto} \quad$}
\begin{tikzpicture}
[scale = 0.6,
line width = 0.8pt,
v/.style = {circle, fill = black, inner sep = 0.8mm},u/.style = {circle, fill = white, inner sep = 0.1mm}]
  \node[v] (1) at (2, 0) {};
  \node[v] (2) at (0, 3.4) {};
  \node[v] (3) at (-2, 0) {};
  \draw[-] (1) to (2);
  \draw[-] (2) to (3);
  \draw[-] (3) to (1);
  \node[u] (11) at (2, 0.35) {};
  \node[u] (22) at (0.25, 3.3) {};
  \draw[draw= blue,->] (11) to (22);
  \node[u] (23) at (-0.25, 3.3) {};
  \node[u] (33) at (-2, 0.35) {};
  \draw[draw= red,->] (33) to (23);
\end{tikzpicture}
\caption{Perfect state transfer does not occur between distinct vertex type states on $K_3$} \label{0809-1}
\end{center}
\end{figure}
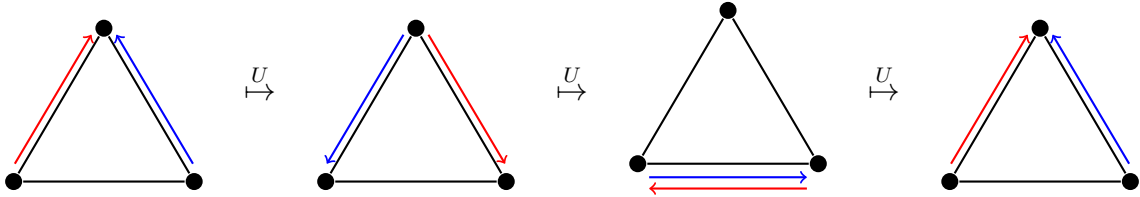

\section*{Acknowledgements}
S.K. is supported by JSPS KAKENHI (Grant No. 20J01175).
E.S. acknowledges financial supports from the Grant-in-Aid of Scientific Research (C) Japan Society
for the Promotion of Science (Grant No. 19K03616) and Research Origin for Dressed Photon.

\end{document}